\documentclass[12pt]{amsart}%
\usepackage{amsfonts}
\usepackage{amsmath}
\usepackage{amssymb}
\usepackage{amscd}
\usepackage{graphicx}%
\providecommand{\U}[1]{\protect\rule{.1in}{.1in}}
\newtheorem{theorem}{Theorem}
\newtheorem{lemma}{Lemma}

\newtheorem{corollary}{Corollary}

\numberwithin{equation}{section}
\begin{document}
\title[Gradient Shrinking Sasaki-Ricci Solitons]{Geometry and Topology of Gradient Shrinking Sasaki-Ricci Solitons}
\author{$^{\ast}$Shu-Cheng Chang}
\address{$^{\ast}$Shanghai Institute of Mathematics and Interdisciplinary Sciences,
Shanghai, 200433, China}
\email{scchang@math.ntu.edu.tw}
\author{$^{\dag}$Yingbo Han}
\address{$^{\dag}${School of Mathematics and Statistics, Xinyang Normal University}\\
Xinyang,464000, Henan, P.R. China}
\email{{yingbohan@163.com}}
\author{$^{\ddag}$Chin-Tung Wu}
\address{$^{\ddag}$Department of Applied Mathematics, National Pingtung University,
Pingtung 90003, Taiwan}
\email{ctwu@mail.nptu.edu.tw }
\thanks{$^{\ast}$Partially supported in part by Startup Foundation for Advanced
Talents of the Shanghai Institute for Mathematics and Interdisciplinary
Sciences (No.2302-SRFP-2024-0049) $^{\ddag}$Research supported in part by NSTC
grant 113-2115-M-153-003, Taiwan. $^{\dag}$Research partially supported by
NSFC 11971415 and NSF of Henan Province 252300421497.}

\begin{abstract}
In this paper, we study the geometry and topology of complete gradient
shrinking Sasaki-Ricci solitons. We first prove that they must be connected at
infinity. This is a Sasaki analogue of gradient shrinking K\"{a}hler-Ricci
solitons. Secondly, with the positive sectional curvature or positive
transverse holomorphic bisectional curvature, we show that they must be
compact. All results are served as a generalization of Perelman in dimension
three, of Naber in dimension four, and of Munteanu-Wang in all dimensions, respectively.

\end{abstract}
\maketitle

\section{Introduction}

It is the very first paper by Smoczyk-Wang-Zhang \cite{swz} to introduce the
Sasaki-Ricci flow on $M\times\lbrack0,T)$%
\[%
\begin{array}
[c]{c}%
\frac{d}{dt}g^{T}(x,t)=-Ric^{T}(x,t)
\end{array}
\]
and to prove the existence theorem of Sasaki $\eta$-Einstein metrics on
Sasakian $(2n+1)$-manifolds $(M,\xi,\eta,g,\Phi)$ when the basic first Chern
class is negative $c_{1}^{B}(M)<0$ or null $c_{1}^{B}(M)=0$. However, it is
wild open when a compact Sasakian $(2n+1)$-manifold is transverse Fano
$c_{1}^{B}(M)>0.$ We refer to \cite{cj}, \cite{chlw}, \cite{cllw} and
\cite{cclw} for later developments along this direction.

On the other hand, Sasaki-Ricci solitons $(M,g^{T},\psi,X)$ are natural
generalizations of $\eta$-Einstein Sasakian $(2n+1)$-manifolds. They are self
similar solutions to the Sasaki-Ricci flow and play a crucial role in the
study of singularities models of the flow. More precisely, we call
$(M,g^{T},\psi,X)$ a gradient Sasaki-Ricci solitons (or a transverse
K\"{a}hler-Ricci soliton) if for a real Hamiltonian basic function $\psi$ with
respect to the Hamiltonian holomorphic vector field $X$ such that
\begin{equation}%
\begin{array}
[c]{c}%
\psi=\sqrt{-1}\eta\left(  X\right)  \text{ \textrm{and} }d\eta(X,\cdot
)=2\sqrt{-1}\overline{\partial}_{B}\psi,
\end{array}
\label{pr1}%
\end{equation}
and the transverse K\"{a}hler metric $g^{T}=(g_{j\overline{k}})$ with
$R_{j\overline{k}}^{T}=R_{j\overline{k}}+2g_{j\overline{k}}$, we have
\[%
\begin{array}
[c]{c}%
R_{j\overline{k}}^{T}+\psi_{j\overline{k}}=(A+2)g_{j\overline{k}}\text{
\textrm{or} }R_{j\overline{k}}+\psi_{j\overline{k}}=Ag_{j\overline{k}}.
\end{array}
\]
It is called expanding, steady and shrinking if for a constant $A<-2;$%
\ $A=-2;$\ $-2<A$, respectively.
Up to
$%
D%
$-homothetic,
one can take
$A=2n$ such that
\begin{equation}%
\begin{array}
[c]{c}%
R_{j\overline{k}}+\psi_{j\overline{k}}=2ng_{j\overline{k}}%
\end{array}
\label{1}%
\end{equation}
is the gradient shrinking Sasaki-Ricci soliton. It is Sasaki-Einstein if
$\psi$ is constant due to $Ric(\xi,\xi)=2n.$

In the present paper, our first objective is to show that the gradient
shrinking Sasaki-Ricci soliton must be connected at infinity. This is a Sasaki
analogue of gradient shrinking K\"{a}hler-Ricci solitons due to Munteanu-Wang
\cite{mw1}.

\begin{theorem}
There exists only one end for a gradient shrinking Sasaki-Ricci soliton
$(M,g^{T},\psi,X)$.
\end{theorem}

For a three dimensional gradient shrinking Ricci soliton with bounded
curvature, its sectional curvature is necessarily nonnegative due to Ivey
\cite{i} and Hamilton \cite{h}. In \cite{p}, Perelman proved that a three
dimensional noncollapsing gradient shrinking Ricci soliton with positive and
bounded sectional curvature must be compact. Later, Naber \cite{n} proved that
any $4$-dimensional gradient shrinking Ricci soliton with bounded nonnegative
curvature operator and positive Ricci curvature must be compact. Recently,
Munteanu-Wang \cite{mw2} generalized the result that a gradient shrinking
Ricci soliton with positive sectional curvature must be compact. Later,
Wu-Zhang \cite{wz} generalized this result to a gradient shrinking
K\"{a}hler-Ricci soliton with positive bisectional curvature is compact.

\begin{theorem}
\label{T1}Let $(M,g^{T},\psi,X)$ be an $(2n+1)$-dimensional gradient shrinking
Sasaki-Ricci soliton with positive sectional curvature or positive transverse
holomorphic bisectional curvature. Then it must be compact.
\end{theorem}

In the paper of B\"{o}hm-Wilking \cite{bw}, they proved that any compact
Riemannian manifold of all dimensions with positive curvature operators must
be space form. As a consequence of Theorem \ref{T1}, we obtain a Sasaki
analogue of gradient shrinking Ricci solitons as follows.

\begin{corollary}
Let $(M,g^{T},\psi,X)$ be an $(2n+1)$-dimensional complete gradient shrinking
Sasaki-Ricci soliton with nonnegative curvature operator and positive Ricci
curvature. Then $(M,g)$ must be a finite quotient of the sphere $\mathbb{S}%
^{2n+1}.$
\end{corollary}

The paper is organized as follows. In section $2$, we prove on a Sasakian
manifold with a given weight function $\varphi$, any $\varphi$-harmonic
function with finite weighted Dirichlet integral is transversely plurihamonic.
In section $3,$ we show that a gradient shrinking Sasaki-Ricci soliton has
only one nonparabolic end. In section $4$, we claim that a gradient shrinking
Sasaki-Ricci soliton with the positive sectional curvature or positive
transverse holomorphic bisectional curvature must be compact.

\section{A vanishing Theorem for $\varphi$-harmonic functions}

In this section, we proof the following Theorem \ref{Thm2}. It generalizes a
result by Munteanu-Wang \cite{mw1} in the K\"{a}hler case, which proved that
on a K\"{a}hler manifold with a given weight function $\varphi$, any $\varphi
$-harmonic function with finite weighted Dirichlet integral is plurihamonic.

\begin{theorem}
\label{Thm2}For a Sasakian manifold $(M,\xi,\eta,g,\Phi)$. We assume that
$\varphi:M\rightarrow\mathbb{R}$ is a basic function which satisfies
$\varphi_{\alpha\beta}=0$. Additionally, there exists a constant $C>0$ such
that
\[%
\begin{array}
[c]{c}%
|\nabla\varphi|\leq C(d(x_{0},x)+1)
\end{array}
\]
for any $x\in M$. Then any real solution $u$ to
\[%
\begin{array}
[c]{c}%
\Delta_{\varphi}u:=\Delta u-\left\langle \nabla u,\nabla\varphi\right\rangle
=0\text{\ \textrm{with}\ }\int_{M}|\nabla u|^{2}e^{-\varphi}<\infty
\end{array}
\]
must be transversely plurihamonic and $u_{00}=0$. Moreover, it is constant if
$\varphi$ is proper on $M$.
\end{theorem}

\begin{proof}
For a fixed point $x_{0}\in M$, let the cut-off function $\phi$ be defined by%
\[%
\begin{array}
[c]{c}%
\phi(x)=\left\{
\begin{array}
[c]{ll}%
1 & \mathrm{on}\text{ }B_{x_{0}}(R)\\
\frac{2R-d(x_{0},x)}{R} & \mathrm{on}\text{ }B_{x_{0}}(2R)\backslash B_{x_{0}%
}(R)\\
0 & \mathrm{on}\text{ }M\backslash B_{x_{0}}(2R).
\end{array}
\right.
\end{array}
\]
We then consider the following integration for second derivatives of $u$
\begin{equation}%
\begin{array}
[c]{ll}
& \int_{M}\left[  |u_{\alpha\bar{\beta}}|^{2}+|u_{0\alpha}|^{2}+|u_{0\bar
{\alpha}}|^{2}+|u_{00}|^{2}\right]  \phi^{2}e^{-\varphi}\\
= & \int_{M}u_{\alpha\bar{\beta}}u_{\bar{\alpha}\beta}\phi^{2}e^{-\varphi
}+u_{0\alpha}u_{0\bar{\alpha}}\phi^{2}e^{-\varphi}+u_{0\bar{\alpha}}%
u_{0\alpha}\phi^{2}e^{-\varphi}+u_{00}u_{00}\phi^{2}e^{-\varphi},
\end{array}
\label{11}%
\end{equation}
here $u_{0}=\xi u.$ We investigate each term in (\ref{11}). Integration by
parts implies
\begin{equation}%
\begin{array}
[c]{ll}
& \int_{M}u_{\alpha\bar{\beta}}u_{\bar{\alpha}\beta}\phi^{2}e^{-\varphi}\\
= & \int_{M}\mathrm{\operatorname{Re}}(u_{\alpha\bar{\beta}}[u_{\beta
\bar{\alpha}}+\sqrt{-1}g_{\alpha\bar{\beta}}u_{0}])\phi^{2}e^{-\varphi}%
=\int_{M}\mathrm{\operatorname{Re}}(u_{\alpha\bar{\beta}}u_{\beta\bar{\alpha}%
})\phi^{2}e^{-\varphi}\\
= & -\int_{M}\mathrm{\operatorname{Re}}(u_{\alpha\bar{\beta}\bar{\alpha}%
}u_{\beta})\phi^{2}e^{-\varphi}-\int_{M}\mathrm{\operatorname{Re}}%
(u_{\alpha\bar{\beta}}u_{\beta}[(\phi^{2})_{\bar{\alpha}}-\varphi_{\bar
{\alpha}}\phi^{2}])e^{-\varphi}\\
= & -\int_{M}\mathrm{\operatorname{Re}}(u_{\alpha\bar{\alpha}\bar{\beta}%
}u_{\beta})\phi^{2}e^{-\varphi}-\int_{M}\mathrm{\operatorname{Re}}%
(u_{\alpha\bar{\beta}}u_{\beta}[(\phi^{2})_{\bar{\alpha}}-\varphi_{\bar
{\alpha}}\phi^{2}])e^{-\varphi}\\
= & \int_{M}\mathrm{\operatorname{Re}}([u_{00}-\Delta u]_{\bar{\beta}}%
u_{\beta})\phi^{2}e^{-\varphi}-\int_{M}\mathrm{\operatorname{Re}}%
(u_{\alpha\bar{\beta}}u_{\beta}[(\phi^{2})_{\bar{\alpha}}-\varphi_{\bar
{\alpha}}\phi^{2}])e^{-\varphi}\\
= & \int_{M}\mathrm{\operatorname{Re}}(u_{00\bar{\beta}}u_{\beta}-(\Delta
u)_{\bar{\beta}}u_{\beta})\phi^{2}e^{-\varphi}-\int_{M}%
\mathrm{\operatorname{Re}}(u_{\alpha\bar{\beta}}u_{\beta}[(\phi^{2}%
)_{\bar{\alpha}}-\varphi_{\bar{\alpha}}\phi^{2}])e^{-\varphi}.
\end{array}
\label{12}%
\end{equation}
here $\Delta u=u_{\alpha\bar{\alpha}}+u_{00}$ and $\mathrm{\operatorname{Re}%
}(z):=\frac{z+\overline{z}}{2}$ denotes the real part of the complex number
$z.$ Integration by parts again
\begin{equation}%
\begin{array}
[c]{ll}
& \int_{M}u_{0\alpha}u_{0\bar{\alpha}}\phi^{2}e^{-\varphi}=\int_{M}%
\mathrm{\operatorname{Re}}(u_{\alpha0}u_{0\bar{\alpha}})\phi^{2}e^{-\varphi}\\
= & -\int_{M}\mathrm{\operatorname{Re}}(u_{0\bar{\alpha}0}u_{\alpha})\phi
^{2}e^{-\varphi}-\int_{M}\mathrm{\operatorname{Re}}(u_{\bar{\alpha}0}%
u_{\alpha})(\phi^{2})_{0}e^{-\varphi}\\
= & -\int_{M}\mathrm{\operatorname{Re}}([u_{00\bar{\alpha}}+u_{\bar{\alpha}%
}]u_{\alpha})\phi^{2}e^{-\varphi}-\int_{M}\mathrm{\operatorname{Re}}%
(u_{\bar{\alpha}0}u_{\alpha})(\phi^{2})_{0}e^{-\varphi},
\end{array}
\label{13}%
\end{equation}
here we used $\varphi$ is basic, that is $\varphi_{0}=0.$ Also
\begin{equation}%
\begin{array}
[c]{ll}
& \int_{M}u_{0\bar{\alpha}}u_{\alpha0}\phi^{2}e^{-\varphi}\\
= & -\int_{M}\mathrm{\operatorname{Re}}(u_{\alpha0\bar{\alpha}})u_{0}\phi
^{2}e^{-\varphi}-\int_{M}\mathrm{\operatorname{Re}}(u_{\alpha0}[(\phi
^{2})_{\bar{\alpha}}-\varphi_{\bar{\alpha}}\phi^{2}])u_{0}e^{-\varphi}\\
= & -\int_{M}u_{\alpha\bar{\alpha}0}u_{0}\phi^{2}e^{-\varphi}-\int
_{M}\mathrm{\operatorname{Re}}(u_{\alpha0}[(\phi^{2})_{\bar{\alpha}}%
-\varphi_{\bar{\alpha}}\phi^{2}])u_{0}e^{-\varphi}.
\end{array}
\label{14}%
\end{equation}
The last term in (\ref{11}) becomes%
\begin{equation}%
\begin{array}
[c]{c}%
\int_{M}u_{00}u_{00}\phi^{2}e^{-\varphi}=-\int_{M}u_{000}u_{0}\phi
^{2}e^{-\varphi}-\int_{M}u_{00}u_{0}(\phi^{2})_{0}e^{-\varphi}.
\end{array}
\label{15}%
\end{equation}
By combining (\ref{12}), (\ref{13}), (\ref{14}) and (\ref{15}) with
(\ref{11}), we obtain%
\begin{equation}%
\begin{array}
[c]{ll}
& \int_{M}\left[  |u_{\alpha\bar{\beta}}|^{2}+|u_{0\alpha}|^{2}+|u_{0\bar
{\alpha}}|^{2}+|u_{00}|^{2}\right]  \phi^{2}e^{-\varphi}\\
\leq & -\int_{M}\left[  \mathrm{\operatorname{Re}}((\Delta u)_{\bar{\beta}%
}u_{\beta}+(\Delta u)_{0}u_{0})\right]  \phi^{2}e^{-\varphi}+\int
_{M}\operatorname{Re}(u_{\alpha\bar{\beta}}u_{\beta}\varphi_{\bar{\alpha}%
})\phi^{2}e^{-\varphi}\\
& +\int_{M}\operatorname{Re}(u_{\alpha0}u_{0}\varphi_{\bar{\alpha}})\phi
^{2}e^{-\varphi}-\int_{M}[\mathrm{\operatorname{Re}}(u_{\bar{\alpha}%
0}u_{\alpha})+u_{00}u_{0}](\phi^{2})_{0}e^{-\varphi}\\
& -\int_{M}\operatorname{Re}((u_{\alpha\bar{\beta}}u_{\beta}+u_{\alpha0}%
u_{0})(\phi^{2})_{\bar{\alpha}})e^{-\varphi}.
\end{array}
\label{16}%
\end{equation}

First, we deal with the first term in (\ref{16}),
\[%
\begin{array}
[c]{ll}
& -\int_{M}\mathrm{\operatorname{Re}}((\Delta u)_{\bar{\beta}}u_{\beta
}+(\Delta u)_{0}u_{0})\phi^{2}e^{-\varphi}=-\int_{M}\left\langle \nabla\Delta
u,\nabla u\right\rangle \phi^{2}e^{-\varphi}\\
= & \int_{M}(\Delta u)(\Delta_{\varphi}u)\phi^{2}e^{-\varphi}+\int_{M}(\Delta
u)\left\langle \nabla u,\nabla\phi^{2}\right\rangle e^{-\varphi}\\
= & \int_{M}(\Delta u)\left\langle \nabla u,\nabla\phi^{2}\right\rangle
e^{-\varphi},
\end{array}
\]
thus by the Cauchy-Schwarz inequality
\begin{equation}%
\begin{array}
[c]{ll}
& -\int_{M}\mathrm{\operatorname{Re}}((\Delta u)_{\bar{\beta}}u_{\beta
}+(\Delta u)_{0}u_{0})\phi^{2}e^{-\varphi}\\
\leq & \frac{1}{8(2n+1)}\int_{M}(\Delta u)^{2}\phi^{2}e^{-\varphi}%
+8(2n+1)\int_{M}|\nabla u|^{2}|\nabla\phi|^{2}e^{-\varphi}\\
\leq & \frac{1}{4}\int_{M}[|u_{\alpha\bar{\beta}}|^{2}+|u_{00}|^{2}]\phi
^{2}e^{-\varphi}+8(2n+1)\int_{M}|\nabla u|^{2}|\nabla\phi|^{2}e^{-\varphi}.
\end{array}
\label{17}%
\end{equation}
Secondly integration by parts yields
\[%
\begin{array}
[c]{ll}
& \int_{M}\operatorname{Re}(u_{\alpha\bar{\beta}}u_{\beta}\varphi_{\bar
{\alpha}})\phi^{2}e^{-\varphi}\\
= & -\int_{M}\operatorname{Re}(u_{\alpha}u_{\beta\overline{\beta}}%
\varphi_{\bar{\alpha}})\phi^{2}e^{-\varphi}-\int_{M}\operatorname{Re}%
(u_{\alpha}u_{\beta}\varphi_{\bar{\alpha}\overline{\beta}})\phi^{2}%
e^{-\varphi}\\
& +\int_{M}\operatorname{Re}(u_{\alpha}u_{\beta}\varphi_{\bar{\alpha}}%
\varphi_{\overline{\beta}})\phi^{2}e^{-\varphi}-\int_{M}\operatorname{Re}%
(u_{\alpha}u_{\beta}\varphi_{\bar{\alpha}}(\phi^{2})_{\overline{\beta}%
})e^{-\varphi},
\end{array}
\]
where we have used the assumption $\varphi_{\alpha\beta}=\varphi_{\bar{\alpha
}\overline{\beta}}=0.$ Thus%
\begin{equation}%
\begin{array}
[c]{ll}
& \int_{M}\operatorname{Re}(u_{\alpha\bar{\beta}}u_{\beta}\varphi_{\bar
{\alpha}})\phi^{2}e^{-\varphi}\\
= & -\int_{M}(\Delta u)\operatorname{Re}(u_{\alpha}\varphi_{\bar{\alpha}}%
)\phi^{2}e^{-\varphi}+\int_{M}\operatorname{Re}(u_{\alpha}\varphi_{\bar
{\alpha}})u_{00}\phi^{2}e^{-\varphi}\\
& +\int_{M}\operatorname{Re}(u_{\alpha}u_{\beta}\varphi_{\bar{\alpha}}%
\varphi_{\overline{\beta}})\phi^{2}e^{-\varphi}-\int_{M}\operatorname{Re}%
(u_{\alpha}u_{\beta}\varphi_{\bar{\alpha}}(\phi^{2})_{\overline{\beta}%
})e^{-\varphi}\\
= & -\int_{M}\langle\nabla u,\nabla\varphi\rangle^{2}\phi^{2}e^{-\varphi}%
+\int_{M}\operatorname{Re}(u_{\alpha}\varphi_{\bar{\alpha}})u_{00}\phi
^{2}e^{-\varphi}\\
& +\int_{M}\operatorname{Re}(u_{\alpha}u_{\beta}\varphi_{\bar{\alpha}}%
\varphi_{\overline{\beta}})\phi^{2}e^{-\varphi}-\int_{M}\operatorname{Re}%
(u_{\alpha}u_{\beta}\varphi_{\bar{\alpha}}(\phi^{2})_{\overline{\beta}%
})e^{-\varphi}.
\end{array}
\label{18}%
\end{equation}
Note that%
\[%
\begin{array}
[c]{c}%
\operatorname{Re}(u_{\alpha}u_{\beta}\varphi_{\bar{\alpha}}\varphi
_{\overline{\beta}})\leq\langle\nabla u,\nabla\varphi\rangle^{2}.
\end{array}
\]
From the first term in line 3 in (\ref{16}), we have
\[%
\begin{array}
[c]{c}%
\int_{M}\operatorname{Re}(u_{\alpha0}u_{0}\varphi_{\bar{\alpha}})\phi
^{2}e^{-\varphi}=-\int_{M}\operatorname{Re}(u_{\alpha}\varphi_{\bar{\alpha}%
})u_{00}\phi^{2}e^{-\varphi}-\int_{M}\mathrm{\operatorname{Re}}(u_{\alpha
}\varphi_{\bar{\alpha}})u_{0}(\phi^{2})_{0}e^{-\varphi},
\end{array}
\]
since $\varphi$ is basic, so $\varphi_{\bar{\alpha}0}=\varphi_{0\bar{\alpha}%
}=0.$ Plugging these into (\ref{18}), we conclude%
\begin{equation}%
\begin{array}
[c]{c}%
\int_{M}\operatorname{Re}(u_{\alpha\bar{\beta}}u_{\beta}\varphi_{\bar{\alpha}%
}+u_{\alpha0}u_{0}\varphi_{\bar{\alpha}})\phi^{2}e^{-\varphi}\leq4\int
_{M}|\nabla u|^{2}|\nabla\varphi|\phi|\nabla\phi|e^{-\varphi}.
\end{array}
\label{19}%
\end{equation}
For the last two terms in (\ref{16}). By the Cauchy-Schwarz inequality again%
\begin{equation}%
\begin{array}
[c]{ll}
& -\int_{M}[\mathrm{\operatorname{Re}}(u_{\bar{\alpha}0}u_{\alpha}%
)+u_{00}u_{0}](\phi^{2})_{0}+\operatorname{Re}((u_{\alpha\bar{\beta}}u_{\beta
}+u_{\alpha0}u_{0})(\phi^{2})_{\bar{\alpha}})]e^{-\varphi}\\
\leq & \frac{1}{4}\int_{M}[|u_{\alpha\bar{\beta}}|^{2}+2|u_{\alpha0}%
|^{2}+2|u_{\bar{\alpha}0}|^{2}+|u_{00}|]\phi^{2}e^{-\varphi}\\
& +c_{1}(n)\int_{M}|\nabla u|^{2}|\nabla\phi|^{2}e^{-\varphi}+c_{2}(n)\int
_{M}|\nabla\varphi||\nabla u|^{2}\phi|\nabla\phi|e^{-\varphi}.
\end{array}
\label{19a}%
\end{equation}
Therefore, by substituting (\ref{17}), (\ref{19}) and (\ref{19a}) into
(\ref{16}), we final get
\begin{equation}%
\begin{array}
[c]{ll}
& \int_{M}\left[  |u_{\alpha\bar{\beta}}|^{2}+|u_{0\alpha}|^{2}+|u_{0\bar
{\alpha}}|^{2}+|u_{00}|^{2}\right]  \phi^{2}e^{-\varphi}\\
\leq & c_{3}(n)\int_{M}|\nabla u|^{2}|\nabla\phi|^{2}e^{-\varphi}+c_{4}%
(n)\int_{M}|\nabla\varphi||\nabla u|^{2}\phi|\nabla\phi|e^{-\varphi}.
\end{array}
\label{19b}%
\end{equation}
Since $\int_{M}|\nabla u|^{2}e^{-\varphi}<\infty$, it implies that%
\[%
\begin{array}
[c]{c}%
\int_{M}|\nabla u|^{2}|\nabla\phi|^{2}e^{-\varphi}\rightarrow0\text{
}\mathrm{as}\text{ }R\rightarrow\infty.
\end{array}
\]
Furthermore, by the assumption that $|\nabla\varphi|\leq C(d(x_{0},x)+1)$, we
have $|\nabla\varphi||\nabla\phi|\leq C$ on $M$. It yields that
\[%
\begin{array}
[c]{c}%
\int_{M}|\nabla\varphi||\nabla u|^{2}\phi|\nabla\phi|e^{-\varphi}%
\rightarrow0\text{ }\mathrm{as}\text{ }R\rightarrow\infty.
\end{array}
\]
We conclude that $u_{\alpha\bar{\beta}}=0,$ $u_{\alpha0}=0$ and $u_{00}=0$ by
letting $R\rightarrow\infty$ in (\ref{19b}).

To show $u$ is constant when $\varphi$ is proper. Note that $\langle\nabla
u,\nabla\varphi\rangle=0$ as $u$ is both $\varphi$-harmonic and harmonic. Let
\[%
\begin{array}
[c]{c}%
D(t)=\{x:\varphi(x)\leq t\},
\end{array}
\]
whish is compact as $\varphi$ is proper. Now we compute
\[%
\begin{array}
[c]{c}%
\int_{D(t)}|\nabla u|^{2}=\frac{1}{2}\int_{D(t)}\Delta u^{2}=\frac{1}{2}%
\int_{\partial D(t)}\frac{\partial u^{2}}{\partial\nu}=\int_{\partial
D(t)}u\frac{\langle\nabla u,\nabla\varphi\rangle}{|\nabla\varphi|}=0.
\end{array}
\]
Since this is true for any $t$, it follows that $|\nabla u|=0$ on $M$, so $u$
is constant.
\end{proof}

\section{Ends of shrinking Sasaki-Ricci solitons}

In this section, we assume $(M,g^{T},\psi,X)$ is a gradient shrinking
Sasaki-Ricci soliton. For a Sasaki-Ricci soliton, the first author, Li and Lin
\cite{cll} showed that the potential function $\psi$ satisfies
\begin{equation}%
\begin{array}
[c]{c}%
n(r(x)-7)_{+}^{2}\leq\psi(x)\leq n(r(x)+\sqrt{3})^{2}%
\end{array}
\label{21}%
\end{equation}
for all $r(x):=d(p,x)$, here $p$ is a minimum point of $\psi$. Using the
Bianchi identities,
\begin{equation}%
\begin{array}
[c]{c}%
R+|\nabla\psi|^{2}=(4n-2)\psi,
\end{array}
\label{21a}%
\end{equation}
where $R$ denotes the scalar curvature of $(M,g)$. Since $R\geq0$ according to
a result showed by Chang-Li-Lin \cite{cll}, we have
\begin{equation}%
\begin{array}
[c]{c}%
|\nabla\psi|^{2}\leq(4n-2)\psi.
\end{array}
\label{22}%
\end{equation}

We now prove the following theorem follows the method adopted in \cite{mw1}.

\begin{theorem}
Let $(M,g^{T},\psi,X)$ be a gradient shrinking Sasaki-Ricci soliton, then
$(M,g)$ has only one end.
\end{theorem}

\begin{proof}
Let $(M,g,e^{-\varphi}dv)$ be a smooth metric space with the weight function
\[
\varphi=-a\psi,
\]
where $a>0$ is a constant. An end $E$ of $(M,g,e^{-\varphi}dv)$ is said to be
$\varphi$-nonparabolic if there exists a positive Green's function for the
weighted Laplacian%
\[
\Delta_{\varphi}u=\Delta u-\left\langle \nabla u,\nabla\varphi\right\rangle
\]
satisfying the Neumann boundary conditions on $\partial E$. Otherwise, it is
called $\varphi$-parabolic.

We first show that $M$ does not admit any $\varphi$-parabolic ends. Suppose
$E$ is a $\varphi$-parabolic end of $M$. A characterize a $\varphi$-parabolic
end, due to Nakai \cite{na}, is by the existence of a proper $\varphi
$-harmonic function $h$ on the end. So we have a function $h\geq1$ defined on
$E$ satisfies
\begin{equation}%
\begin{array}
[c]{c}%
\Delta_{\varphi}h=\Delta h+a\left\langle \nabla h,\nabla\psi\right\rangle =0
\end{array}
\label{23a}%
\end{equation}
such that%
\begin{equation}%
\begin{array}
[c]{c}%
\lim_{x\rightarrow E(\infty)}h(x)=\infty\text{ }\mathrm{and}\text{ }h=1\text{
}\mathrm{on}\text{ }\partial E.
\end{array}
\label{23}%
\end{equation}
The goal is to show that (\ref{23}) leads to a contradiction, which implies
that all ends of $(M,g)$ are $\varphi$-nonparabolic.

For $t>1$ and $b>c>1$, we denote
\[%
\begin{array}
[c]{l}%
l(t):=\{x\in E:h(x)=t\}\\
L(c,b):=\{x\in E:c<h(x)<b\},
\end{array}
\]
which are compact since $h$ is proper. By the Stokes Theorem, one get
\[%
\begin{array}
[c]{c}%
0=\int_{L(c,b)}(\Delta_{\varphi}h)e^{-\varphi}=\int_{l(b)}\frac{\partial
h}{\partial\nu}e^{-\varphi}-\int_{l(c)}\frac{\partial h}{\partial\nu
}e^{-\varphi}=\int_{l(b)}|\nabla h|e^{-\varphi}-\int_{l(c)}|\nabla
h|e^{-\varphi},
\end{array}
\]
where we used $\frac{\partial}{\partial\nu}=\frac{\nabla h}{|\nabla h|}$ is
the unit normal to the level set of $h$. This implies that $\int_{l(t)}|\nabla
h|e^{-\varphi}$ is independent of $t\geq1$. Then by apply the co-area formula
to get
\[%
\begin{array}
[c]{c}%
\int_{E}|\nabla\ln h|^{2}e^{-\varphi}=\int_{L(1,\infty)}|\nabla\ln
h|^{2}e^{-\varphi}=\int_{1}^{\infty}\int_{l(t)}\frac{|\nabla h|}{h^{2}%
}e^{-\varphi}dt=\int_{1}^{\infty}\frac{1}{t^{2}}dt\int_{l(t_{0})}|\nabla
h|e^{-\varphi}%
\end{array}
\]
which is finite. This implies that $\int_{E}|\nabla\ln h|^{2}e^{-\varphi
}<\infty$ and for any $x\in E$ with $B_{x}(1)\subset E$, we have by (\ref{21})%
\begin{equation}%
\begin{array}
[c]{c}%
\int_{B_{x}(1)}|\nabla\ln h|^{2}\leq ce^{-na(r(x)-7)^{2}}.
\end{array}
\label{24}%
\end{equation}
Now we transform (\ref{24}) into a pointwise estimate. We denote $v:=\ln h$
and by the Bochner formula%
\[%
\begin{array}
[c]{c}%
\frac{1}{2}\Delta|\nabla v|^{2}=|\mathrm{Hess}(v)|^{2}+\left\langle
\nabla\Delta v,\nabla v\right\rangle +Ric(\nabla v,\nabla v).
\end{array}
\]
For any vector field $V$,%
\[%
\begin{array}
[c]{lll}%
Ric(V,V) & = & Ric(V_{D}+\eta(V)\xi,V_{D}+\eta(V)\xi)\\
& = & Ric^{T}(V_{D},V_{D})-2g(V_{D},V_{D})+2n\eta(V)^{2}\\
& = & (2n+2)|V_{D}|^{2}-\mathrm{Hess}_{D}(\psi)(V_{D},V_{D})-2g(V_{D}%
,V_{D})+2n\eta(V)^{2}\\
& = & 2n|V|^{2}-\mathrm{Hess}(\psi)(V,V).
\end{array}
\]
Since by (\ref{23a}),%
\begin{equation}%
\begin{array}
[c]{c}%
\Delta v=\frac{1}{h}\Delta h-\frac{1}{h^{2}}|\nabla h|^{2}=-a\left\langle
\nabla v,\nabla\psi\right\rangle -|\nabla v|^{2},
\end{array}
\label{25}%
\end{equation}
we conclude that%
\[%
\begin{array}
[c]{lll}%
\frac{1}{2}\Delta|\nabla v|^{2} & = & |\mathrm{Hess}(v)|^{2}-a\left\langle
\nabla\left\langle \nabla v,\nabla\psi\right\rangle ,\nabla v\right\rangle
-\left\langle \nabla|\nabla v|^{2},\nabla v\right\rangle \\
&  & +2n|\nabla v|^{2}-\mathrm{Hess}(\psi)(\nabla v,\nabla v)\\
& \geq & \frac{1}{2}|\mathrm{Hess}(v)|^{2}-(4n-2)a^{2}\psi|\nabla
v|^{2}-\left\langle \nabla|\nabla v|^{2},\nabla v\right\rangle \\
&  & +2n|\nabla v|^{2}-(a+1)\mathrm{Hess}(\psi)(\nabla v,\nabla v).
\end{array}
\]
By the Cauchy-Schwarz inequality and (\ref{25}), we obtain%
\[%
\begin{array}
[c]{lll}%
|\mathrm{Hess}(v)|^{2} & \geq & \frac{1}{2n+1}(\Delta v)^{2}=\frac{1}%
{2n+1}\left(  a\left\langle \nabla v,\nabla\psi\right\rangle +|\nabla
v|^{2}\right)  ^{2}\\
& \geq & \frac{1}{2(2n+1)}|\nabla v|^{4}-4a^{2}\psi|\nabla v|^{2}.
\end{array}
\]
In conclusion, denote $\sigma:=|\nabla v|^{2},$ we have the following
inequality%
\begin{equation}%
\begin{array}
[c]{c}%
\frac{1}{4(2n+1)}\sigma^{2}\leq(4n+2)a^{2}\psi\sigma+\left\langle \nabla
\sigma,\nabla v\right\rangle +(a+1)\mathrm{Hess}(\psi)(\nabla v,\nabla
v)+\frac{1}{2}\Delta\sigma.
\end{array}
\label{26}%
\end{equation}
Multiplying the above inequality by $\phi^{2}\sigma^{p-1}$ and integrating
over $M$, where $\phi$ is a cut-off function with support in $B_{x}(1)$ and
$p\geq c(n)>0$ be large enough depending only on dimension $2n+1$, we get%
\begin{equation}%
\begin{array}
[c]{ll}
& \frac{1}{4(2n+1)}\int_{M}\sigma^{p+1}\phi^{2}\\
\leq & (4n+2)a^{2}\int_{M}\psi\sigma^{p}\phi^{2}+\int_{M}\left\langle
\nabla\sigma,\nabla v\right\rangle \sigma^{p-1}\phi^{2}\\
& +(a+1)\int_{M}\mathrm{Hess}(\psi)(\nabla v,\nabla v)\sigma^{p-1}\phi
^{2}+\frac{1}{2}\int_{M}\sigma^{p-1}(\Delta\sigma)\phi^{2}.
\end{array}
\label{27}%
\end{equation}
Integrating by parts, we get%
\begin{equation}%
\begin{array}
[c]{ll}
& p\int_{M}\left\langle \nabla\sigma,\nabla v\right\rangle \sigma^{p-1}%
\phi^{2}\\
= & \int_{M}\left\langle \nabla\sigma^{p},\nabla v\right\rangle \phi^{2}%
=-\int_{M}\sigma^{p}\Delta v\phi^{2}-\int_{M}\sigma^{p}\left\langle \nabla
v,\nabla\phi^{2}\right\rangle \\
= & \int_{M}\sigma^{p}(a\left\langle \nabla v,\nabla\psi\right\rangle
+\sigma)\phi^{2}-\int_{M}\sigma^{p}\left\langle \nabla v,\nabla\phi
^{2}\right\rangle \\
\leq & 2\int_{M}\sigma^{p+1}\phi^{2}+2\int_{M}\sigma^{p}|\nabla\phi
|^{2}+(2n-1)a^{2}\int_{M}\psi\sigma^{p}\phi^{2}.
\end{array}
\label{28}%
\end{equation}
Also,
\begin{equation}%
\begin{array}
[c]{lll}%
\int_{M}\sigma^{p-1}(\Delta\sigma)\phi^{2} & = & -(p-1)\int_{M}\sigma
^{p-2}|\nabla\sigma|^{2}\phi^{2}-\int_{M}\sigma^{p-1}\left\langle \nabla
\sigma,\nabla\phi^{2}\right\rangle \\
& \leq & -(p-2)\int_{M}\sigma^{p-2}|\nabla\sigma|^{2}\phi^{2}+\int_{M}%
\sigma^{p-1}|\nabla\phi|^{2}.
\end{array}
\label{29}%
\end{equation}
Integrating by parts again,
\begin{equation}%
\begin{array}
[c]{ll}
& \int_{M}\mathrm{Hess}(\psi)(\nabla v,\nabla v)\sigma^{p-1}\phi^{2}\\
= & -\int_{M}\mathrm{Hess}(v)(\nabla v,\nabla\psi)\sigma^{p-1}\phi^{2}%
-\int_{M}\sigma^{p-1}\left\langle \nabla v,\nabla\psi\right\rangle
\left\langle \nabla v,\nabla\phi^{2}\right\rangle \\
& -(p-1)\int_{M}\left\langle \nabla v,\nabla\psi\right\rangle \left\langle
\nabla v,\nabla\sigma\right\rangle \sigma^{p-2}\phi^{2}-\int_{M}\Delta
v\left\langle \nabla v,\nabla\psi\right\rangle \sigma^{p-1}\phi^{2}.
\end{array}
\label{30}%
\end{equation}
We now study each term on the right hand side of the above equation
(\ref{30}).
\[%
\begin{array}
[c]{lll}
&  & -\int_{M}\mathrm{Hess}(v)(\nabla v,\nabla\psi)\sigma^{p-1}\phi^{2}\\
& = & -\frac{1}{2}\int_{M}\left\langle \nabla|\nabla v|^{2},\nabla
\psi\right\rangle \sigma^{p-1}\phi^{2}=-\frac{1}{2p}\int_{M}\left\langle
\nabla\sigma^{p},\nabla\psi\right\rangle \sigma^{p-1}\phi^{2}\\
& = & \frac{1}{2p}\int_{M}\Delta\psi\sigma^{p}\phi^{2}+\frac{1}{2p}\int
_{M}\left\langle \nabla\phi^{2},\nabla\psi\right\rangle \sigma^{p}\\
& \leq & \frac{4n-1}{p}\int_{M}\psi\sigma^{p}\phi^{2}+\frac{1}{2p}\int
_{M}\sigma^{p}|\nabla\phi|^{2},
\end{array}
\]
here we used $\Delta\psi=2n(2n+1)-R\leq2n(2n+1)\leq4n\psi$ and (\ref{22}).
And,%
\[%
\begin{array}
[c]{lll}%
-\int_{M}\sigma^{p-1}\left\langle \nabla v,\nabla\psi\right\rangle
\left\langle \nabla v,\nabla\phi^{2}\right\rangle  & \leq & 2\int_{M}%
|\nabla\psi||\nabla\phi|\sigma^{p}\phi\\
& \leq & (4n-2)\int_{M}\psi\sigma^{p}\phi^{2}+\int_{M}\sigma^{p}|\nabla
\phi|^{2}.
\end{array}
\]
The third term in the right hand side of (\ref{30}) is bounded by%
\[%
\begin{array}
[c]{ll}
& -(p-1)\int_{M}\left\langle \nabla v,\nabla\psi\right\rangle \left\langle
\nabla v,\nabla\sigma\right\rangle \sigma^{p-2}\phi^{2}\leq p\int_{M}%
|\nabla\psi||\nabla\sigma|\sigma^{p-1}\phi\\
\leq & \frac{p-2}{4(a+1)}\int_{M}|\nabla\sigma|^{2}\sigma^{p-2}\phi
^{2}+2(4n-2)p(a+1)\int_{M}\psi\sigma^{p}\phi^{2}.
\end{array}
\]
The last term in (\ref{30}) is%
\[%
\begin{array}
[c]{lll}%
-\int_{M}\Delta v\left\langle \nabla v,\nabla\psi\right\rangle \sigma
^{p-1}\phi^{2} & = & \int_{M}(a\left\langle \nabla v,\nabla\psi\right\rangle
-|\nabla v|^{2})\left\langle \nabla v,\nabla\psi\right\rangle \sigma^{p-1}%
\phi^{2}\\
& = & n(4n-2)p(a+1)\int_{M}\psi\sigma^{p}\phi^{2}+\frac{1}{np(a+1)}\int
_{M}\sigma^{p+1}\phi^{2}.
\end{array}
\]
Plugging all these estimates into (\ref{30}), we have%
\begin{equation}%
\begin{array}
[c]{ll}
& (a+1)\int_{M}\mathrm{Hess}(\psi)(\nabla v,\nabla v)\sigma^{p-1}\phi^{2}\\
= & c_{1}(n)(a+1)^{2}p\int_{M}\psi\sigma^{p}\phi^{2}+\frac{1}{np}\int
_{M}\sigma^{p+1}\phi^{2}\\
& +\frac{p-2}{4}\int_{M}|\nabla\sigma|^{2}\sigma^{p-2}\phi^{2}+2(a+1)\int
_{M}\sigma^{p}|\nabla\phi|^{2}.
\end{array}
\label{31}%
\end{equation}
Using (\ref{28}), (\ref{29}) and (\ref{31}), we get from (\ref{27}) that for
$p$ large enough and depending on $2n+1,$
\begin{equation}%
\begin{array}
[c]{ll}
& (p-2)\int_{M}|\nabla\sigma|^{2}\sigma^{p-2}\phi^{2}\\
\leq & c_{2}(n)(a+1)^{2}p\int_{M}\psi\sigma^{p}\phi^{2}+c_{3}(n)(a+1)\int
_{M}\sigma^{p}|\nabla\phi|^{2}.
\end{array}
\label{32}%
\end{equation}
However, using (\ref{21}), we have a Sobolev inequality (equation (2.14) in
\cite{mw1})
\begin{equation}%
\begin{array}
[c]{c}%
\left(  \int_{B_{x}(1)}\chi^{2\mu}\right)  ^{\frac{1}{\mu}}\leq Ce^{c(n)r(x)}%
\left(  \int_{B_{x}(1)}|\nabla\chi|^{2}+C\int_{B_{x}(1)}\chi^{2}\right)  .
\end{array}
\label{33}%
\end{equation}
for some constant $\mu>1$ and for all smooth function $\chi$ with support in
$B_{x}(1)$, where $r(x):=d(p,x).$ Using (\ref{32}) and (\ref{33}), by the
standard De Giorgi-Nash-Moser iteration, we can obtain
\[%
\begin{array}
[c]{c}%
\sigma(x)\leq C(a)e^{c(n)r(x)}\int_{B_{x}(1)}\sigma.
\end{array}
\]
Then by combining this inequality with (\ref{24}), one concludes%
\[%
\begin{array}
[c]{c}%
|\nabla\mathrm{\ln}h(x)|\leq\sqrt{\sigma(x)}\leq ce^{-\frac{n}{2}%
a(r(x)-7)^{2}}.
\end{array}
\]
By integrating the above estimate for $h$ along minimizing geodesics we see
that $h$ must be bounded on the end $E$, which contradicts with (\ref{23}).
This proves all ends of $(M,g)$ are $\varphi$-nonparabolic.

To finish the proof of the Theorem, we assume to the contrary that $(M,g)$ has
more than one $\varphi$-nonparabolic end. Li-Tam \cite{lt} showed that there
exists a nonconstant $\varphi$-harmonic function $u:M\rightarrow\mathbb{R}$
such that
\[%
\begin{array}
[c]{c}%
\Delta_{\varphi}u=0\text{ \textrm{and} }\int_{M}|\nabla u|^{2}e^{-\varphi
}<\infty
\end{array}
\]
But $M$ as a gradient shrinking Sasaki-Ricci soliton satisfies all the
assumptions of Theorem \ref{Thm2} with $\varphi=-a\psi$. Indeed, the
definition of $\psi$ in (\ref{pr1}) implies that $\varphi$ is basic and
$\varphi_{\alpha\beta}=0$. Also (\ref{21}) and (\ref{22}) shows $\varphi$ is
proper and $|\nabla\varphi|$ grows at most linearly in the distance function.
Therefore, we conclude that $u$ must be a constant. This contradiction shows
that $M$ has only one end.
\end{proof}

\section{Positively curved shrinking Sasaki-Ricci solitons}

In this section, we show that a gradient shrinking Sasaki-Ricci soliton
$(M,g^{T},\psi,X)$ with the positive sectional curvature or positive
transverse holomorphic bisectional curvature must be compact. The proof
follows the method adopted in \cite{mw2}. We define%
\[%
\begin{array}
[c]{c}%
D(r):=\{x\in M:\rho(x)=2\sqrt{\psi(x)}<r\}.
\end{array}
\]
Then we have the following estimate that the average of the scalar curvature
over $D(r)$ is bounded by $2n(2n+1).$

\begin{lemma}
For all $r>0$, we have%
\begin{equation}%
\begin{array}
[c]{c}%
\int_{D(r)}R\leq2n(2n+1)\mathrm{Vol}(D(r)).
\end{array}
\label{40}%
\end{equation}

\end{lemma}

\begin{proof}
Taking the trace in (\ref{1}), we have%
\[
R+\Delta_{B}\psi=2n(2n+1).
\]
Thus the Stokes Theorem implies
\[%
\begin{array}
[c]{ll}
& 2n(2n+1)\mathrm{Vol}(D(r)-\int_{D(r)}R=\int_{D(r)}\Delta_{B}\psi\\
= & \int_{\partial D(r)}\langle\nabla\psi,\frac{\nabla\rho}{|\nabla\rho
|}\rangle=\int_{\partial D(r)}|\nabla\psi|\geq0,
\end{array}
\]
where we used $\frac{\nabla\rho}{|\nabla\rho|}=\frac{\nabla\psi}{|\nabla\psi
|}$ is the unit normal to $\partial D(r).$
\end{proof}

\begin{theorem}
\label{Thm}Let $(M,g^{T},\psi,X)$ be a gradient shrinking Sasaki-Ricci soliton
with the positive sectional curvature. Then $(M,g)$ must be compact.
\end{theorem}

\begin{proof}
Assume that $M$ is noncompact. Recall that the Ricci curvature $R_{ij}$ of
$(M,g)$ satisfies the following differential equation ((iv) in equation (6.1)
in \cite{cll})
\[
\Delta_{B,\psi}R_{ij}=4n(R_{ij}-g_{ij})-2R_{kl}R_{ikjl},
\]
where $\Delta_{B,\psi}:=\Delta_{B}-\left\langle \nabla\psi,\cdot\right\rangle
$ is the weighted basic Laplacian and $R_{ikjl}$ the curvature tensor
restricted to the horizontal distribution $\Gamma(D),$ here $D=\ker\eta$. Our
$R_{ikjl}$ is the opposite of the one taken in \cite{cll}. Note that
\begin{equation}
\Delta_{B,\psi}(R_{ij}-g_{ij})=(4n-2)(R_{ij}-g_{ij})-2(R_{kl}-g_{kl}%
)R_{ikjl},\label{41}%
\end{equation}
since $\Sigma_{k=1}^{2n}R_{ikjk}=R_{ij}-g_{ij}.$ We then consider the tensor
$Ric_{D}-g^{T}$ on $\Gamma(D)$. For $v\in\Gamma(D),$ one have%
\[%
\begin{array}
[c]{c}%
(R_{ij}-g_{ij})v_{i}v_{j}=R_{ikjk}v_{i}v_{j}=\mathrm{Rm}(v,e_{k}%
,v,e_{k})=\Sigma_{k=1}^{2n}K(v,e_{k})>0
\end{array}
\]
where $\mathrm{Rm}$ denotes the Riemann curvature tensor and $K(v,w)$ is the
sectional curvature of the plane spanned by $v$ and $w$. The inequality above
is true because we assumed $K(v,w)$ is positive for any $v$, $w$. Let
$\lambda(x)>0$ be the smallest eigenvalue of $Ric_{D}-g^{T}$ and suppose
$v\in\Gamma(D)$ is an eigenvector corresponding to $\lambda(x).$ We note that%
\[%
\begin{array}
[c]{c}%
(R_{kl}-g_{kl})R_{ikjl}v_{i}v_{j}=(R_{kl}-g_{kl})\mathrm{Rm}(v,e_{k},v,e_{l})
\end{array}
\]
Diagonalizing the Ricci tensor minus transverse K\"{a}hler metric so that
$R_{kl}-g_{kl}=\lambda_{k}\delta_{kl}$, it follows from above that%
\begin{equation}%
\begin{array}
[c]{c}%
(R_{kl}-g_{kl})R_{ikjl}v_{i}v_{j}=\Sigma_{k=1}^{2n}K(v,e_{k})\lambda_{k}\geq0.
\end{array}
\label{42}%
\end{equation}
From (\ref{41}) and (\ref{42}) it follows that $\lambda$ satisfies the
following differential inequality in the sense of barriers
\begin{equation}
\Delta_{B,\psi}\lambda\leq(4n-2)\lambda.\label{43}%
\end{equation}

We now adapt an idea from \cite{cly} to obtain a global lower bound for
$\lambda$ by using maximum principle. Since $\Delta_{B,\psi}\psi
=-2(2n-1)\psi+2n(2n+1).$ So
\[%
\begin{array}
[c]{lll}%
\Delta_{B,\psi}\frac{1}{\psi} & = & 2(2n-1)\frac{1}{\psi}-\frac{1}{\psi^{2}%
}[2n(2n+1)-2\frac{\left\vert \nabla\psi\right\vert ^{2}}{\psi}]\\
& \geq & 2(2n-1)\frac{1}{\psi}-2n(2n+1)\frac{1}{\psi^{2}}%
\end{array}
\]
and
\[%
\begin{array}
[c]{lll}%
\Delta_{B,\psi}\frac{1}{\psi^{2}} & = & 4(2n-1)\frac{1}{\psi^{2}}-\frac
{1}{\psi^{3}}[4n(2n+1)-6\frac{\left\vert \nabla\psi\right\vert ^{2}}{\psi}]\\
& \geq & 4(2n-1)\frac{1}{\psi^{2}}-4n(2n+1)\frac{1}{\psi^{3}}.
\end{array}
\]
Therefore, using these inequalities and combining with (\ref{43}), the
function%
\[%
\begin{array}
[c]{c}%
\varphi:=\lambda-\frac{b}{\psi}-\frac{nb}{\psi^{2}},
\end{array}
\]
for a small positive constant $b$, will satisfy the differential inequality
\begin{equation}%
\begin{array}
[c]{l}%
\Delta_{B,\psi}\varphi\leq2(2n-1)\varphi-\frac{2n(6n-5)b}{\psi^{2}}%
+\frac{16n^{2}(2n+1)b}{\psi^{3}}.
\end{array}
\label{44}%
\end{equation}
Take $b$ to be sufficiently small such that $\varphi$ is positive on
$B_{y}(10\sqrt{n}+5),$ where $y$ is a minimum point of $\psi$. Due to the
inequality (\ref{21}) for $\psi$, we may suppose that $\varphi$ attains the
global minimum at $p$ outside the ball $B_{y}(10\sqrt{n}+5)$. So if
$\varphi(p)<0$, then, at the minimum point $p$, we get
\[%
\begin{array}
[c]{l}%
\psi(p)\leq\frac{8n(2n+1)}{6n-5}\leq24n,
\end{array}
\]
by the inequality (\ref{44}). However, it contradicts with the lower bound
estimate of $\psi(p)$ i.e. $\psi(p)\geq25n,$ by the inequality (\ref{21})
again. Thus, we conclude that $\varphi\geq0$ on $M$. So there exists some
constant $0<b\leq1$ such that%
\begin{equation}%
\begin{array}
[c]{c}%
Ric_{D}-g^{T}\geq\frac{b}{\psi}%
\end{array}
\label{45}%
\end{equation}
on $M.$ We use this to show that the scalar curvature on $M$ must satisfy%
\begin{equation}%
\begin{array}
[c]{c}%
R\geq b\ln\psi.
\end{array}
\label{46}%
\end{equation}
Suppose by contradiction that there exists a point $x\in M$ such that%
\[%
\begin{array}
[c]{c}%
R(x)\leq b\ln\psi(x).
\end{array}
\]
Let us consider $\sigma(\eta)$, where $\eta\geq0$, to be the integral curve of
$-\frac{\nabla\psi}{|\nabla\psi|^{2}}$ , such that $\sigma(0)=x$. As
$|\nabla\psi|(x)>0$ by (\ref{21a}), this flow is defined at least in a
neighborhood of $x$. Since $\frac{d}{d\eta}\psi(\sigma(\eta))=-1$, it results
that $\psi(\sigma(\eta))=t-\eta$, where $t:=\psi(x)$. Using (\ref{45}), we
have that
\[%
\begin{array}
[c]{c}%
\frac{d}{d\eta}R(\sigma(\eta))=-\frac{\left\langle \nabla R,\nabla
\psi\right\rangle }{|\nabla\psi|^{2}}=-2(Ric_{D}-g^{T})(\frac{\nabla\psi
}{|\nabla\psi|},\frac{\nabla\psi}{|\nabla\psi|})\leq-\frac{2b}{\psi
(\sigma(\eta))}=-\frac{2b}{t-\eta}.
\end{array}
\]
where we have used the relation that $\nabla R=2(Ric_{D}-g^{T})(\nabla\psi).$
Integrating this differential inequality, we get
\begin{equation}%
\begin{array}
[c]{c}%
R(x)-R(\sigma(\eta))\geq2b\ln t-2b\ln(t-\eta).
\end{array}
\label{47}%
\end{equation}
Since initially $R(x)\leq b\ln t$, it follows that%
\[%
\begin{array}
[c]{c}%
R(\sigma(\eta))\leq2b\ln(t-\eta)-b\ln t<b\ln\psi(\sigma(\eta)).
\end{array}
\]
Since $R(y)\leq b\ln\psi(y)$ holds true along $\sigma(\eta)$ and $b\leq1$,
(\ref{21a}) implies that $|\nabla\psi|(\sigma(\eta))\geq\sqrt{4n-2}$ for all
$0\leq\eta\leq t-1$. Hence, $\sigma(\eta)$ exists at least for $0\leq\eta\leq
t-1$, and
\[%
\begin{array}
[c]{c}%
R(\sigma(t-1))<b\ln\psi(\sigma(t-1))=0,
\end{array}
\]
which is a contradiction, since the scalar curvature $R$ is nonnegative. In
conclusion, (\ref{46}) is true for all $x$.

Now from (\ref{40}), for all $r>0$, we have%
\[%
\begin{array}
[c]{c}%
\int_{B_{p}(r)}R\leq C(n)\mathrm{Vol}(B_{p}(r)).
\end{array}
\]
For any $q$ with $d(p,q)=\frac{3}{4}r$, it follows from (\ref{46}) and
(\ref{21}) that
\begin{equation}%
\begin{array}
[c]{c}%
\int_{B_{p}(r)}R\geq\int_{B_{q}(\frac{r}{4})}R\geq b\ln[n(r(x)-7)_{+}%
^{2}]\mathrm{Vol}(B_{q}(\frac{r}{4})).
\end{array}
\label{48}%
\end{equation}
The Bishop-Gromov relative volume comparison implies that%
\begin{equation}%
\begin{array}
[c]{c}%
c(n)\mathrm{Vol}(B_{q}(\frac{r}{4}))\geq\mathrm{Vol}(B_{q}(2r))\geq
\mathrm{Vol}(B_{p}(r)).
\end{array}
\label{49}%
\end{equation}
By (\ref{48}) and (\ref{49}) we infer that $\ln r\leq\frac{c(n)}{b}$, which is
a contradiction to $M$ being noncompact.
\end{proof}

In this following, we show that a gradient shrinking Sasaki-Ricci soliton with
the positive transverse holomorphic bisectional curvature must be compact. The
argument is the same as the proof of Theorem \ref{Thm}. We let $(M,g^{T}%
,\psi,X)$ be a gradient shrinking Sasaki-Ricci soliton such that
\begin{equation}%
\begin{array}
[c]{c}%
R_{i\overline{j}}+\psi_{i\overline{j}}=2ng_{i\overline{j}},\text{ \textrm{and}
}\psi_{ij}=0.
\end{array}
\label{2a}%
\end{equation}

\begin{lemma}
For a Sasakian manifold $(M,\xi,\eta,g,\Phi)$. Let $f$ be a basic real
function, then
\begin{equation}%
\begin{array}
[c]{c}%
\Delta_{B}f_{i\overline{j}}=\nabla_{\overline{j}}\nabla_{i}\Delta
_{B}f-R_{i\overline{j}k\overline{l}}f_{l\overline{k}}+\frac{1}{2}%
R_{i\overline{k}}f_{k\overline{j}}+\frac{1}{2}R_{k\overline{j}}f_{i\overline
{k}}-f_{i\overline{j}}.
\end{array}
\label{2}%
\end{equation}

\end{lemma}

\begin{proof}
By the covariant differentiation commutation formula, we obtain%
\[%
\begin{array}
[c]{lll}%
\nabla_{\overline{k}}\nabla_{k}f_{i\overline{j}} & = & \nabla_{k}%
\nabla_{\overline{k}}f_{i\overline{j}}+R_{k\overline{k}i\overline{p}%
}f_{p\overline{j}}-R_{k\overline{k}p\overline{j}}f_{i\overline{p}}\\
& = & \nabla_{k}\nabla_{\overline{k}}f_{i\overline{j}}+(R_{i\overline{p}%
}-g_{i\overline{p}})f_{p\overline{j}}-(R_{p\overline{j}}-g_{p\overline{j}%
})f_{i\overline{p}}\\
& = & \nabla_{k}\nabla_{\overline{k}}f_{i\overline{j}}+R_{i\overline{p}%
}f_{p\overline{j}}-R_{p\overline{j}}f_{i\overline{p}},
\end{array}
\]
and%
\[%
\begin{array}
[c]{lll}%
\nabla_{k}\nabla_{\overline{k}}f_{i\overline{j}} & = & \nabla_{k}%
\nabla_{\overline{j}}f_{i\overline{k}}=\nabla_{\overline{j}}\nabla
_{k}f_{i\overline{k}}-R_{k\overline{j}i\overline{p}}f_{p\overline{k}%
}+R_{k\overline{j}p\overline{k}}f_{i\overline{p}}\\
& = & \nabla_{\overline{j}}\nabla_{i}f_{k\overline{k}}-R_{i\overline
{j}k\overline{p}}f_{p\overline{k}}+(R_{p\overline{j}}-g_{p\overline{j}%
})f_{i\overline{p}}.
\end{array}
\]
Hence,%
\[%
\begin{array}
[c]{lll}%
\Delta_{B}f_{i\overline{j}} & = & \frac{1}{2}(\nabla_{k}\nabla_{\overline{k}%
}+\nabla_{\overline{k}}\nabla_{k})f_{i\overline{j}}\\
& = & \nabla_{\overline{j}}\nabla_{i}\Delta_{B}f-R_{i\overline{j}k\overline
{l}}f_{l\overline{k}}+\frac{1}{2}(R_{i\overline{k}}f_{k\overline{j}%
}+R_{k\overline{j}}f_{i\overline{k}})-f_{i\overline{j}}.
\end{array}
\]

\end{proof}

Then we have the following formula for the Ricci curvature $R_{i\overline{j}}$.

\begin{lemma}
\label{lemma}Let $(M,g^{T},\psi,X)$ be a gradient shrinking Sasaki-Ricci
soliton, then%
\[
\Delta_{B,\psi}R_{i\overline{j}}=2n(R_{i\overline{j}}-g_{i\overline{j}%
})-R_{l\overline{k}}R_{i\overline{j}k\overline{l}}.
\]
where $\Delta_{B,\psi}$ is the weighted basic Laplacian defined by
$\Delta_{B,\psi}R_{i\overline{j}}=\Delta_{B}R_{i\overline{j}}-g^{k\overline
{l}}\nabla_{k}\psi\nabla_{\overline{l}}R_{i\overline{j}}.$
\end{lemma}

\begin{proof}
Using the equation of the gradient shrinking Sasaki-Ricci soliton (\ref{2a})
and (\ref{2}) with $\psi=f$, we obtain%
\[%
\begin{array}
[c]{lll}%
\Delta_{B,\psi}R_{i\overline{j}} & = & -\Delta_{B}\psi_{i\overline{j}}%
-\psi_{k}\nabla_{\overline{k}}R_{i\overline{j}}\\
& = & -\nabla_{\overline{j}}\nabla_{i}\Delta_{B}\psi+R_{i\overline
{j}k\overline{l}}\psi_{l\overline{k}}-\frac{1}{2}(R_{i\overline{k}}%
\psi_{k\overline{j}}+R_{k\overline{j}}\psi_{i\overline{k}})+\psi
_{i\overline{j}}-\psi_{k}\nabla_{\overline{k}}R_{i\overline{j}}\\
& = & \nabla_{\overline{j}}(R_{i\overline{k}}\psi_{k}-\psi_{i})+R_{i\overline
{j}k\overline{l}}\psi_{l\overline{k}}-\frac{1}{2}(R_{i\overline{k}}%
\psi_{k\overline{j}}+R_{k\overline{j}}\psi_{i\overline{k}})+\psi
_{i\overline{j}}-\psi_{k}\nabla_{\overline{k}}R_{i\overline{j}}\\
& = & \nabla_{\overline{j}}R_{i\overline{k}}\psi_{k}+R_{i\overline
{j}k\overline{l}}\psi_{l\overline{k}}+\frac{1}{2}(R_{i\overline{k}}%
\psi_{k\overline{j}}-R_{k\overline{j}}\psi_{i\overline{k}})-\psi_{k}%
\nabla_{\overline{k}}R_{i\overline{j}}\\
& = & R_{i\overline{j}k\overline{l}}\psi_{l\overline{k}}+\frac{1}%
{2}(R_{i\overline{k}}\psi_{k\overline{j}}-R_{k\overline{j}}\psi_{i\overline
{k}})\\
& = & R_{i\overline{j}k\overline{l}}(2ng_{l\overline{k}}-R_{l\overline{k}%
})+\frac{1}{2}[R_{i\overline{k}}(2ng_{k\overline{j}}-R_{k\overline{j}%
})-R_{k\overline{j}}(2ng_{i\overline{k}}-R_{i\overline{k}})]\\
& = & 2n(R_{i\overline{j}}-g_{i\overline{j}})-R_{l\overline{k}}R_{i\overline
{j}k\overline{l}}%
\end{array}
\]
where we have used $\nabla_{\overline{j}}R_{i\overline{k}}=-\nabla
_{\overline{j}}\psi_{i\overline{k}}=-\nabla_{\overline{k}}\psi_{i\overline{j}%
}=\nabla_{\overline{k}}R_{i\overline{j}}$ in the fifth equality.
\end{proof}

\begin{theorem}
Let $(M,g^{T},\psi,X)$ be a gradient shrinking Sasaki-Ricci soliton with
positive transverse holomorphic bisectional curvature. Then $(M,g)$ must be compact.
\end{theorem}

\begin{proof}
From Lemma \ref{lemma}, we have
\begin{equation}
\Delta_{B,\psi}(R_{i\overline{j}}-g_{i\overline{j}})=(2n-1)(R_{i\overline{j}%
}-g_{i\overline{j}})-(R_{l\overline{k}}-g_{l\overline{k}})R_{i\overline
{j}k\overline{l}},\label{51}%
\end{equation}
and for all nonzero vector $v\in\Gamma(D)$
\[%
\begin{array}
[c]{c}%
(R_{i\overline{j}}-g_{i\overline{j}})v^{i}v^{\overline{j}}=\Sigma
_{k}R_{i\overline{j}k\overline{k}}v^{i}v^{\overline{j}}=\Sigma_{k}%
\mathrm{Rm}(v,\overline{v},e_{k},e_{\overline{k}})>0
\end{array}
\]
if the transverse holomorphic bisectional curvature is positive. Denote
$\mu(x)$ as the smallest positive eigenvalue of $Rc_{D}-g^{T}$ on $\Gamma(D)$
and let $v\in\Gamma(D)$ is an eigenvector corresponding to $\mu(x),$ then%
\[%
\begin{array}
[c]{c}%
(R_{l\overline{k}}-g_{l\overline{k}})R_{i\overline{j}k\overline{l}}%
v^{i}v^{\overline{j}}=(R_{l\overline{k}}-g_{l\overline{k}})\mathrm{Rm}%
(v,\overline{v},e_{k},e_{\overline{l}}).
\end{array}
\]
Diagonalizing the Ricci tensor minus transverse K\"{a}hler metric so that
$R_{l\overline{k}}-g_{l\overline{k}}=\mu_{k}\delta_{lk}$. Since the transverse
holomorphic bisectional curvature is nonnegative, we have%
\begin{equation}%
\begin{array}
[c]{c}%
(R_{l\overline{k}}-g_{l\overline{k}})R_{i\overline{j}k\overline{l}}%
v^{i}v^{\overline{j}}=\Sigma_{k}\mathrm{Rm}(v,\overline{v},e_{k}%
,e_{\overline{k}})\mu_{k}\geq0.
\end{array}
\label{52}%
\end{equation}
From (\ref{51}) and (\ref{52}), we see that $\mu$ satisfies the following
differential inequality in the sense of barriers
\begin{equation}
\Delta_{B,\psi}\mu\leq(2n-1)\mu.\label{53}%
\end{equation}

As above we can get a global lower bound for $\mu$ by using maximum principle.
Since
\[%
\begin{array}
[c]{lll}%
\Delta_{B,\psi}\frac{1}{\sqrt{\psi}} & = & (2n-1)\frac{1}{\sqrt{\psi}}%
-\frac{1}{\sqrt{\psi3}}[n(2n+1)-\frac{3}{4}\frac{\left\vert \nabla
\psi\right\vert ^{2}}{\psi}]\\
& \geq & (2n-1)\frac{1}{\sqrt{\psi}}-n(2n+1)\frac{1}{\sqrt{\psi3}}%
\end{array}
\]
and
\[%
\begin{array}
[c]{lll}%
\Delta_{B,\psi}\frac{1}{\sqrt{\psi^{3}}} & = & 3(2n-1)\frac{1}{\sqrt{\psi^{3}%
}}-\frac{1}{\sqrt{\psi^{5}}}[3n(2n+1)-\frac{15}{4}\frac{\left\vert \nabla
\psi\right\vert ^{2}}{\psi}]\\
& \geq & 3(2n-1)\frac{1}{\sqrt{\psi^{3}}}-3n(2n+1)\frac{1}{\sqrt{\psi^{5}}}.
\end{array}
\]
Therefore, using these inequalities and combining with (\ref{53}), the
function%
\[%
\begin{array}
[c]{c}%
\phi:=\mu-\frac{b}{\sqrt{\psi}}-\frac{2nb}{\sqrt{\psi3}},
\end{array}
\]
for a sufficiently small positive constant $b$, will satisfy the differential
inequality
\[%
\begin{array}
[c]{l}%
\Delta_{B,\psi}\phi\leq(2n-1)\phi-\frac{n(6n-5)b}{\sqrt{\psi^{3}}}%
+\frac{6n^{2}(2n+1)b}{\sqrt{\psi^{5}}}.
\end{array}
\]
Which will get that $\phi\geq0$ on $M$ by the lower bound estimate for the
potential function $\psi$. So one obtain%
\[%
\begin{array}
[c]{c}%
Rc_{D}-g^{T}\geq\frac{b}{\sqrt{\psi}}\geq\frac{b}{\psi}%
\end{array}
\]
on $M$ for some positive constant $b$ and the scalar curvature will satisfy
$R\geq b\ln\psi$ on $M$. Then, as the same arguments in the proof of Theorem
\ref{Thm}, we can get that $M$ is compact.
\end{proof}


\begin{thebibliography}{9999}                                                                                             %


\bibitem[BW]{bw}C. B\"{o}hm and B. Wilking, \textit{Manifolds with positive
curvature operators are space forms}, Ann. Math. 167 (2008), 1079-1097.

\bibitem[CCLW]{cclw}D.-C. Chang, S.-C. Chang, C. Lin and C.-T. Wu,
\textit{Analytic foliation divisorial contraction on Sasakian manifolds of
dimension five}, arXiv: 2203.01736.

\bibitem[CHLW]{chlw}S.-C. Chang, Y.-B. Han, C. Lin and C.-T. Wu, $L^{4}%
$\textit{-Bound of the transverse Ricci curvature under the Sasaki-Ricci
flow}, J. Math. Study 58 (2025), no. 1, 36-59.

\bibitem[CJ]{cj}T. Collins and A. Jacob, \textit{On the convergence of the
Sasaki-Ricci flow}, Analysis, complex geometry, and mathematical physics: in
honor of Duong H. Phong, 11-21, Contemp. Math., 644, Amer. Math. Soc.,
Providence, RI, 2015.

\bibitem[CLL]{cll}S.-C. Chang, F. Li and C. Lin, \textit{Geometry of shrinking
Sasaki-Ricci solitons I: fundamental equations and characterization of
rigidity}, arXiv:2502.16148.

\bibitem[CLLW]{cllw}S.-C. Chang, F. Li, C. Lin and C.-T. Wu, \textit{On the
existence of conic Sasaki-Einstein metrics on Log Fano Sasakian manifolds}, arXiv:2406.16430.

\bibitem[CLY]{cly}B. Chow, P. Lu and B. Yang, \textit{A lower bound for the
scalar curvature of noncompact nonflat Ricci shrinkers}, C. R. Math. Acad.
Sci. Paris 349 (2011), 1265-1267.

\bibitem[H]{h}R. Hamilton, \textit{The formation of singularities in the Ricci
flow,} Surveys in Differential Geom. 2 (1995), 7-136, International Press.

\bibitem[I]{i}T. Ivey, \textit{Ricci solitons on compact three-manifolds},
Differential Geom. Appl. 3 (1993), no.4, 301-307.

\bibitem[LT]{lt}P. Li and L.F. Tam, \textit{Harmonic functions and the
structure of complete manifolds}, J. Differ. Geom. 35 (1992) 359-383.

\bibitem[N]{n}A. Naber, \textit{Noncompact shrinking }$4$\textit{-solitons
with nonnegative curvature}, J. Reine Angew. Math., 645 (2010), 125-153.

\bibitem[Na]{na}M. Nakai,\textit{ On Evans potentia}l, Proc. Japan Acad. 38
(1962), 624-629.

\bibitem[MW1]{mw1}O. Munteanu and J. Wang, \textit{Topology of K\"{a}hler
Ricci solitons}, J. Differ. Geom. 100 (2015), no. 1, 109-128.

\bibitem[MW2]{mw2}O. Munteanu and J. Wang, \textit{Positively curved shrinking
Ricci solitons are compact}, J. Differ. Geom. 106 (2017), 499-505.

\bibitem[P]{p}G. Perelman, \textit{Ricci flow with surgery on three-manifolds}%
, arXiv:math/0303109.

\bibitem[SWZ]{swz}K. Smoczyk, G. Wang and Y. Zhang, \textit{The Sasaki-Ricci
flow}, Internat. J. Math. 21 (2010), no. 7, 951-969.

\bibitem[WZ]{wz}G. Wu and S. Zhang, \textit{Remarks on shrinking gradient
K\"{a}ler-Ricci solitons with positive bisectional curvature}, C. R. Acad.
Sci. Paris, Ser. I 354 (2016) 713-716.
\end{thebibliography}
\end{document}